\documentclass[11pt,letterpaper, notitlepage]{amsart}

\usepackage{amsmath, amsfonts, amsthm, amssymb}
\usepackage[utf8]{inputenc} 
\usepackage[T1]{fontenc}   
\usepackage{graphicx}

\usepackage[table]{xcolor}
\usepackage{stmaryrd}
\usepackage{bm}
\usepackage{bbm}
\usepackage{tikz}
\usepackage{tikz-cd}
\usepackage[shortlabels]{enumitem}
\setlist[1]{wide}
\setlist[2]{leftmargin=15mm}
\setlist[enumerate]{label=\rm{(\roman*)}}
\setlist[itemize]{label=\raisebox{0.25ex}{\tiny$\bullet$}}

\usepackage[backref, colorlinks, linktocpage, citecolor = blue, linkcolor = blue]{hyperref}

\usepackage{soul}
\usepackage{mathtools}

\usepackage{lmodern}
\usepackage[noadjust]{cite}
\usepackage{framed}
\usepackage{stackrel}
\usepackage{array}
\usepackage{xifthen}

\usepackage[english, french]{babel}
\usepackage{csquotes}

\DeclareQuoteStyle[american]{english}
{\itshape\textquotedblleft}
[\textquotedblleft]
{\textquotedblright}
[0.05em]
{\textquoteleft}
{\textquoteright}

\usepackage{booktabs}
\usepackage{longtable}

\usepackage[all, cmtip]{xy}

\makeatletter
\renewcommand\subsection{\@startsection{subsection}{1}%
  \z@{.5\linespacing\@plus.7\linespacing}{-.5em}%
  {\bfseries}}
\def\@seccntformat#1{%
  \protect\textup{\protect\@secnumfont
    \ifnum\pdfstrcmp{subsection}{#1}=0 \bfseries\fi
    \csname the#1\endcsname
    \protect\@secnumpunct
  }%
}  
\@namedef{subjclassname@2020}{\textup{2020} Mathematics Subject Classification}
\makeatother

\newcommand{\newabstract}[1]{%
	\par\bigskip
	\csname otherlanguage*\endcsname{#1}%
	\csname captions#1\endcsname
	\item[\hskip\labelsep\scshape\abstractname.]
}

\numberwithin{equation}{section}

\newtheorem{theorem}{Theorem}[section]
\newtheorem{proposition}[theorem]{Proposition}
\newtheorem{lemma}[theorem]{Lemma}

\newtheorem{question}[theorem]{Question}
\theoremstyle{definition}
\newtheorem{remark}[theorem]{Remark}
\newtheorem{definition}[theorem]{Definition}

\newtheorem{example}[theorem]{Example}
\newtheoremstyle{customNumber}
     {}          
     {}          
     {\itshape}  
     {}          
     {\bfseries} 
     {.}         
     { }         
     {\thmname{#1}\thmnumber{ #2}\thmnote{ #3}}
\theoremstyle{customNumber}

\renewcommand{\phi}{\varphi}

\DeclareMathOperator{\Q}{\mathbb{Q}}
\DeclareMathOperator{\C}{\mathbb{C}}
\DeclareMathOperator{\A}{\mathbb{A}}

\DeclareMathOperator{\R}{\mathbb{R}}
\DeclareMathOperator{\Z}{\mathbb{Z}}
\DeclareMathOperator{\F}{\mathbb{F}}
\DeclareMathOperator{\PP}{\mathbb{P}}

\DeclareMathOperator{\Spec}{\mathrm{Spec}}
\DeclareMathOperator{\Pic}{\mathrm{Pic}}
\DeclareMathOperator{\id}{\mathrm{id}}
\DeclareMathOperator{\Aut}{\mathrm{Aut}}
\DeclareMathOperator{\Gal}{\mathrm{Gal}}
\DeclareMathOperator{\GL}{\mathrm{GL}}
\DeclareMathOperator{\Proj}{\mathrm{Proj}}

\newcommand{\easy}{\operatorname}
\newcommand{\RomanNumeralCaps}[1]
{\MakeUppercase{\romannumeral #1}}

\title[A rational affine surface with uncountably many real forms]
{A smooth complex rational affine surface with uncountably many real forms \\ ~ \\Une surface affine rationnelle complexe lisse avec un ensemble ind\'enombrable de formes r\'eelles}

\author{Anna Bot}
\address{}

\keywords{Rational surfaces, affine surfaces, real forms, real structures. \\ The author acknowledges support by the Swiss National Science Foundation Grant ``Geo-metrically ruled surfaces'' 200020-192217.} 

\begin{document}

\selectlanguage{english}
\begin{abstract}
	We exhibit a smooth complex rational affine surface with uncountably many nonisomorphic real forms.
\newabstract{french}
	Nous pr\'esentons une surface affine rationnelle complexe lisse avec un ensemble ind\'enombrable de formes r\'eelles non isomorphes.
\end{abstract}

\thanks{}
\maketitle

\tableofcontents

\section{Introduction}

The study of real forms of a complex algebraic variety has seen substantial progress: the finiteness of isomorphism classes over $\mathbb{R}$ of real forms has been proven for example for abelian varieties~\cite{MR181643, MR678890}, projective algebraic surfaces of Kodaira dimension greater than or equal to one~\cite[Appendix D]{MR1795406}, minimal projective algebraic surfaces~\cite[Appendix D]{MR1795406}, del~Pezzo~surfaces~\cite{MR1954070} and compact hyperkähler manifolds~\cite{MR4023392}. Lesieutre constructed in \cite{MR3773792} the first example of a complex projective variety with infinitely many nonisomorphic real forms, and the construction inspired the examples found by Dinh and Oguiso~\cite{MR3934593} of complex projective varieties in every dimension~$d\geq 2$ of Kodaira dimension~$d-2$, and the ones by Dinh, Oguiso and Yu~\cite{MR4444125} of smooth rational projective varieties of dimension greater or equal to three.

Using a field-theoretic approach, \cite{labinet} proves that any projective variety can have at most countably many real forms. We give a short proof in the particular case of rational projective surfaces in Proposition~\ref{prop: at most countably for projective}. Yet, in the affine case, a priori no restrictions hold on the number of real forms. However, such varieties may only arise in dimension at least~$2$ --- see Proposition~\ref{prop: curves finitely many real forms} for a proof of why smooth affine curves admit at most finitely many real forms. Dubouloz, Freudenburg and Moser-Jauslin~\cite{MR4234102} constructed rational affine varieties of dimension greater or equal to four with at least countably infinitely many pairwise nonisomorphic real forms. This prompts the question: can there exist a rational affine surface with uncountably many pairwise nonisomorphic real forms? We can answer in the affirmative, and even describe a moduli space which parametrises an uncountable collection of nonisomorphic real forms of the suitable affine surface:

\begin{theorem}\label{thm: new main theorem}
	There exists a real affine variety $\mathcal{X}$ and a flat morphism $p \colon \mathcal{X} \rightarrow \A^1_{\R} \smallsetminus \{0,1\}$ defined over $\R$ satisfying the following properties:
	\begin{enumerate}[leftmargin=*]
		\item The closed fibres of $p_{\C}\colon \mathcal{X}_{\C} \rightarrow \A^1_{\C} \smallsetminus \{0,1\}$ are pairwise isomorphic smooth complex rational affine surfaces. \label{item: main theorem, 1}
		\item The fibres of $p$ over $\R$-rational closed points are pairwise birationally diffeomorphic smooth real rational affine surfaces. \label{item: main theorem, 2}
		\item Two fibres $p^{-1}(\alpha)$, $p^{-1}(\beta)$ with $\alpha, \beta \in \R \smallsetminus \{0,1\}$ are isomorphic if and only if $\alpha\beta=1$ or $\alpha =\beta$. \label{item: main theorem, 3}
	\end{enumerate}
\end{theorem}

Statement~\ref{item: main theorem, 2} implies that the real forms appearing in this moduli space arise for algebraic reasons only and not because of some topological obstructions. Also, choosing only $\Q$-rational closed points $\alpha \in \A^1_{\R}(\Q)$, we find infinitely (yet countably!) many $\Q$-forms using the above construction; the same holds for any subfield of $\R$. We further point out that the theory and the proofs would also work for a real closed field $\mathbf{k}$ and its algebraic closure $\overline{\mathbf{k}}$ in lieu of $\R$ and $\C$, to the effect that the classes of nonisomorphic $\mathbf{k}$-forms constitute a set of cardinality $|\mathbf{k}|$.

The question whether there exists a smooth complex rational projective surface with an infinite number of isomorphism classes of real forms was posed for example by~\cite[pages 232-233]{MR1795406}, \cite[Problem, page 1128]{MR3473660}, \cite[page 943]{MR3934593} or \cite[{Question 1.5}]{MR4444125}, and an unboundedness result was given in~\cite{bot2021real}. Dinh, Oguiso and Yu~\cite{MR4529417} provided an example of a smooth complex rational projective surface with infinitely many pairwise nonisomorphic real forms, using a nifty construction involving the Kummer K3 surface associated to the product of two elliptic curves. With this, rational surfaces with as many pairwise nonisomorphic real forms as possible have been found, both in the affine and the projective case.

The complex rational affine surface constituting a closed fibre $p_{\C}^{-1}(\alpha)$ of the fibration in Theorem~\ref{thm: new main theorem}~\ref{item: main theorem, 1} with $\alpha \in \C \smallsetminus \{0,1\}$ can be given explicitly as  the blow-up of $\A^2_{\C}$ in the points $(0,0)$, $(1,0)$, $(\alpha, 0)$, $(0,1)$ and $(0,\alpha)$, with the strict transforms of the lines $x=0$ and $y=0$ removed, see Section~\ref{subsec: first}. Such an affine surface can be completed by a linear chain of smooth rational curves and is thus called a \emph{Gizatullin surface} (see Section~\ref{section: moduli}) --- in this case, the curves are the strict transforms of $x=0$ and $y=0$ as well as the line at infinity. In fact, the results in this text were first obtained using the description of the automorphism group of the Gizatullin surface given in~\cite{MR3340326} by Blanc and Dubouloz.

Most likely, the real forms we exhibit in this article are by far not all one can construct, so it would be interesting to determine:

\begin{question}
	Are there real forms of a closed fibre of $p_{\C}$ other than $p^{-1}(\alpha)$ for $\alpha \in \R \smallsetminus \{0,1\}$, and can one find some not birationally diffeomorphic to $p^{-1}(\alpha)$? 
\end{question}

As real forms are a special case of twisted forms, the following questions pose themselves.
\begin{question}
	Fix a field $\mathbf{k}$ with separable closure $\mathbf{k}_s$. Is it possible to construct infinitely/uncountably many affine varieties $X$ defined over $\mathbf{k}$ and pairwise nonisomorphic over $\mathbf{k}$ such that their base changes $X \times_{\Spec(\mathbf{k})} \Spec(\mathbf{k}_s)$ are all isomorphic over $\mathbf{k}_s$?
\end{question}
\begin{question}
	What about the special case of $\mathbf{k}=\Q_p$? What about $\mathbf{k}=\C(t)$?
\end{question}

The paper is structured as follows: to prove Theorem~\ref{thm: new main theorem}, we proceed by introducing some preliminary notions as well as the correspondence between complex schemes with a real structure and real schemes in Section~\ref{section: preliminaries}. Then, we describe the moduli space in Section~\ref{section: moduli}, with two isomorphic descriptions of closed fibres in~\ref{subsec: first} and~\ref{subsec: second}. Subsequently, Sections~\ref{section: closed fibres are pairwise isomorphic},~\ref{section: birational diffeomorphisms} and~\ref{section: nonisomorphic real forms} are devoted to proving~\ref{item: main theorem, 1},~\ref{item: main theorem, 2} and~\ref{item: main theorem, 3} of Theorem~\ref{thm: new main theorem}, respectively.

\textbf{Acknowledgements.} I would like to express my thanks to my PhD advisor J\'er\'emy Blanc for introducing me to this topic and discussing it with me. Moreover, I am very grateful to Marat Gizatullin for his thought-provoking comments and pointing out some carelessness in the first version of this text. Then, my thanks go to Adrien Dubouloz and Johannes Huisman for pointing out the moduli interpretation. Lastly, I am indebted to the referees whose reports helped to drastically simplify the exposition and allow the paper to be self-contained.

\section{Preliminaries} \label{section: preliminaries}

\subsection{Underlying definitions and results}
For any $\R$-scheme $X_0$, denote by $(X_0)_{\C}=X_0\times_{\R} \Spec(\C)$ the \textit{complexification} of $X_0$, where the morphism of schemes $\Spec(\C)\rightarrow \Spec(\R)$ is induced by the Galois extension $\C/\R$. Similarly, if a morphism $\psi \colon X_0 \rightarrow Y_0$ of $\R$-schemes is given, then we denote by $\psi_{\C} \colon (X_0)_{\C} \rightarrow (Y_0)_{\C}$ the base change from $\R$ to $\C$, and refer to it as the \textit{complexification} of $\psi$. 

A \textit{real form} of a complex scheme $X$ is a real scheme $X_0$ together with a $\mathbb{C}$-isomorphism $(X_0)_{\C} \overset{\sim}{\rightarrow} X$. Instead of real forms, one may study real structures, as the two are closely linked. A \textit{real structure} on a complex scheme $X$ is an anti-regular involution $\rho: X \rightarrow X$; by anti-regularity we mean that the diagram
\begin{center}
	\begin{tikzcd}
		X \arrow[r, "\rho"] \arrow[d]  &  X \arrow[d] \\
		\easy{Spec}(\C) \arrow[r, "z \mapsto \overline{z}"]  & \easy{Spec}(\C)
	\end{tikzcd}
\end{center}
commutes. Given two real structures $\rho$ on $X$ and $\rho'$ on $X'$, a morphism $\theta \colon X \rightarrow X'$ is called \emph{real} if $\theta \circ \rho = \rho' \circ \theta$, and we write $\theta \colon (X, \rho) \rightarrow (X', \rho')$. The real structures $\rho$ on $X$ and $\rho'$ on $X'$ are \textit{equivalent} if there exists an isomorphism $\psi\colon (X, \rho) \overset{\sim}{\rightarrow} (X', \rho')$.

\begin{remark} \label{remark: blow-up}
	By the universal property of the blow-up, whenever we blow up a $\C$-scheme $Y$ with real structure $\rho'$ in a closed subscheme invariant under $\rho'$, then there exists a real structure $\rho$ on the blow-up $\pi\colon X \rightarrow Y$ such that $\pi$ is real, see for example~\cite[\RomanNumeralCaps{2}.7.14]{MR0463157}.
\end{remark}
 
On the one hand, considering a real form $X_0$ of $X$ with $\varphi \colon (X_0)_{\C} \overset{\sim}{\rightarrow}X$, we can set $\rho = \varphi\circ\rho_{X_0}\circ\varphi^{-1}$, where $\rho_{X_0} = \easy{id}_{X_0} \times \easy{Spec}(z \mapsto \overline{z})$ is the \emph{canonical real structure} on $(X_0)_{\C}$. On the other hand, suppose we have a real structure $\rho$ on $X$ such that $X$ is covered by $\rho$-stable affine open subsets. It then follows from~\cite[{\RomanNumeralCaps{8}, Cor.~$7.7$}]{MR0217087} that there exists a real scheme $Y$ with underlying set $X/\langle \rho \rangle$ such that $(Y_{\C}, \rho_{Y}) \cong (X, \rho)$. In fact,  \cite[{\RomanNumeralCaps{8}, Cor.~$7.7$}]{MR0217087} provides an equivalence between the category of quasi-projective complex schemes with a real structure and the category of quasi-projective real schemes, where $\R$-isomorphic quasi-projective real schemes correspond to equivalent real structures.

Let $X$ be a quasi-projective $\C$-scheme and fix a real structure $\rho$ on $X$. We then write $(X, \rho)$ for the corresponding $\R$-scheme; we only omit the explicit mention of the real structure when it is clear that we mean the canonical real structure. Denote by $Z^1(\Gal(\C/\R),\Aut_{\C}(X))$ the set consisting of automorphisms $\tau \in \Aut_{\C}(X)$ satisfying $(\tau\circ\rho)^2=\id$ and define the first Galois cohomology pointed set by $H^1(\Gal(\C/\R),\Aut_{\C}(X))=Z^1(\Gal(\C/\R),\Aut_{\C}(X))/\sim$, where $\tau_1 \sim \tau_2$ if and only if there exists $\nu \in \Aut_{\C}(X)$ with $\nu\circ\tau_1=\tau_2\circ\rho\circ\nu\circ\rho$. There exists a one-to-one correspondence between the elements of this Galois cohomology set and real forms:

\begin{theorem}[\cite{MR181643}] \label{thm: Borel and Serre Galois cohomology}
	Let $X$ be a quasi-projective complex scheme and $\rho$ a real structure on $X$. The isomorphism classes of real forms of $X$ are in bijection with the elements of $H^1(\Gal(\C/\R),\Aut_{\C}(X))$, where the nontrivial element of $\Gal(\C/\R)$ acts on $\Aut_{\C}(X)$ by conjugation with $\rho$.
\end{theorem}

Some real forms can be distinguished by a topological analysis: Given an $\R$-scheme $(X,\rho)$, we write $X(\R)=X(\C)^\rho$ for the $\C$-rational points which are fixed by $\rho$ and call it the \textit{real locus} of $(X, \rho)$. When $(X, \rho)$ is a real algebraic variety, the real locus can be endowed with a Euclidean topology as a subspace of some $\R^n$, so that when $(X, \rho)$ is smooth, $X(\R)$ becomes a differentiable manifold. If $(X, \rho)$ and $(\widetilde{X}, \widetilde{\rho})$ are both smooth real algebraic varieties, we say that $(X, \rho)$ and $(\widetilde{X}, \widetilde{\rho})$ are \textit{birationally diffeomorphic} if there exists a real birational map $(X, \rho) \overset{\sim}{\dashrightarrow} (\widetilde{X}, \widetilde{\rho})$ defined at every point of $X(\R)$ and restricting to a diffeomorphism from $X(\R)$ to $\tilde{X}(\R)$.

\begin{example}
	Along the lines of for example \cite[Prop.~1.1]{MR1954070} or \cite[Thm.~4.1]{benzerga}, one can show that if $n$ is even, there is, up to equivalence, only one real structure on $\PP^n_{\C}$, namely $\rho_{\PP^n_{\C}}\colon [ z_0 : \ldots : z_n ] \mapsto [ \overline{z_0} : \ldots : \overline{z_n} ]$, which we also call the standard or classical real structure on $\PP^n_{\C}$. If, however, $n=2k+1$ is odd, there are two, up to equivalence: the standard one and $[z_0 : z_1 : \ldots : z_n] \mapsto [-\overline{z_1} : \overline{z_0}: \ldots : -\overline{z_{2k+1}} : \overline{z_{2k}} ]$. 
	
	It is much harder to determine the equivalence classes of real structures on affine $n$-spaces: On the affine line, there is only one equivalence class containing the complex conjugation, and due to Kambayashi \cite{MR369380}, the real structures on $\A^2_{\C}$ are all equivalent to $(x_1,x_2) \mapsto (\overline{x_1}, \overline{x_2})$. However, already for $\A^n_{\C}$ with $n\geq 3$ it is not known if there is more than the classical real form.
\end{example}

\subsection{Small dimensions}

With Theorem~\ref{thm: Borel and Serre Galois cohomology}, we can see why the possible number of real forms jumps from finitely many in the case of curves to at most countably many in the case of projective rational surfaces. 

\begin{proposition} \label{prop: curves finitely many real forms}
	Any smooth complex quasi-projective curve $C$ has finitely many real forms up to isomorphism.
\end{proposition}
\begin{proof}
	Any smooth quasi-projective curve is either projective or affine. Consider first a smooth complex projective curve $C$ and denote by $g(C)$ its genus. If $g(C)=0$, then $C \cong \PP^1_{\C}$, which, by for example Proposition~$1.1$ in Russo~\cite{MR1954070}, has precisely two real structures up to equivalence, implying that the first Galois cohomology set contains only two elements. If $g(C)=1$, then $C$ is an abelian variety, for which by~\cite{MR181643, MR678890} we find that $H^1(\Gal(\C/\R), \Aut_{\C}(C))$ is finite. Now, for $g(C) \geq 2$, by the Riemann-Hurwitz formula, Natanzon~\cite{MR509374} or Bujalance, Gromadzki and Singerman~\cite{MR1165047}, the automorphism group $\Aut_{\C}(C)$ is finite, and thus $H^1(\Gal(\C/\R), \Aut_{\C}(C))$ is, too. This proves the claim for the projective case.
	
	Let $C$ be a smooth complex affine curve embedded in some $\A^n_{\C}$. Take an embedding $\A^n_{\C} \hookrightarrow \PP^n_{\C}$ and write $\widetilde{C}$ for the normalisation of the projective closure of $C$ in $\PP^n_{\C}$. Denote by $j\colon C \hookrightarrow \widetilde{C}$ the open embedding. Any automorphism $f\colon C\overset{\sim}{\rightarrow} C$ induces a birational map $j\circ f \circ j^{-1}$ on $\widetilde{C}$, with inverse $j\circ f^{-1} \circ j^{-1}$. Since $\widetilde{C}$ is smooth and projective, $j\circ f \circ j^{-1}$ is actually an automorphism. Thus, $\Aut_{\C}(C)  \subseteq \Aut_{\C}(\widetilde{C})$. Moreover, since $j\circ f \circ j^{-1}$ restricts to an automorphism on $C$, it leaves the finite set $\widetilde{C} \smallsetminus C$ invariant.
	
	If $g(\widetilde{C})\geq 2$, the finiteness of $H^1(\Gal(\C/\R), \Aut_{\C}(C))$ follows from the finiteness of $\Aut_{\C}(\widetilde{C})$. When $g(\widetilde{C})=1$, we can specify a point on $\widetilde{C}$, turning $\widetilde{C}$ into an elliptic curve. Then any automorphism of $\widetilde{C}$ is of the form $z \mapsto az+b$ for $a\in \C^*$ and $b\in C$, and leaves the set $\widetilde{C} \smallsetminus C$ invariant. There can only be finitely many such automorphisms, proving that $\Aut_{\C}(C)$ and thus $H^1(\Gal(\C/\R), \Aut_{\C}(C))$ are finite. The last case is $g(\widetilde{C})=0$: then $\widetilde{C} \cong \PP^1_{\C}$ and there are only finitely many automorphisms leaving a subset of at least three points invariant. Thus, we only have to consider $C\cong \A^1_{\C}$ or $C\cong \A^1_{\C} \smallsetminus \{0\}$. The affine line has only one real form up to isomorphism, which can be seen by a direct calculation, and $\A^1_{\C} \smallsetminus \{0\}$ has three, see Proposition~$4.2$ of~\cite{MR4609022}. This completes the proof. 
\end{proof}

To describe the situation for projective rational surfaces, note that a direct consequence of Theorem~\ref{thm: Borel and Serre Galois cohomology} is that if $\Aut_{\C}(X)$ is countable, then so must be the set isomorphism classes of real forms of $X$. 

\begin{proposition} \label{prop: at most countably for projective}
	Let $X$ be a smooth complex rational projective surface. Then $X$ has at most countably many nonisomorphic real forms.
\end{proposition}
\begin{proof}
	Since $X$ is a smooth complex rational projective surface, it is either a blow-up of $\PP^2_{\C}$ or of a Hirzebruch surface $\F_r$ with $r=0$ or $r\geq 2$. Thus, the Picard group of $X$ is some direct sum $\Pic(X) \cong \Z^n$ generated by the classes of the exceptional curves arising from the blow-ups involved, and the total transform of the generators of $\Pic(\PP^2_{\C})$ or $\Pic(\F_r)$, respectively. Therefore, $\Aut(\Pic(X))$ is a subgroup of some $\GL_n(\Z)$. Consider the action of $\Aut_{\C}(X)$ on the Picard group $\Pic(X)$ and denote the image of $\Aut_{\C}(X)$ in $\Aut(\Pic(X))$ by $\Aut_{\C}^{\ast}(X)$. This gives the short exact sequence $1 \rightarrow \Aut_{\C}^{\#}(X)\rightarrow \Aut_{\C}(X) \rightarrow \Aut_{\C}^{\ast}(X) \rightarrow 1$, where $\Aut_{\C}^{\#}(X)$ is the kernel of $\Aut_{\C}(X) \rightarrow \Aut_{\C}^{\ast}(X)$. 
	
	Thanks to Harbourne~\cite[{Cor.~(1.4)}]{MR875372}, at most one of $\Aut_{\C}^{\ast}(X)$ and $\Aut_{\C}^{\#}(X)$ can be infinite. As a subgroup of $\Aut(\easy{Pic}(X))$, the group $\Aut_{\C}^{\ast}(X)$ is always countable; if it is infinite, then $\Aut_{\C}^{\#}(X)$ is finite, and thus $\Aut_{\C}(X)$ is countable. Therefore, $H^1(\Gal(\mathbb{C}/\mathbb{R}), \Aut_{\C}(X))$ is countable, too. If $\Aut_{\C}^{\ast}(X)$ is finite, then the action of any automorphism of $X$ on $\Pic(X)$ is necessarily of finite order. But that means that the logarithm of the spectral radius of the induced map on $\Pic(X) \otimes_{\Z}\R$ has to be $0$, meaning any automorphism of $X$ has zero entropy. In this situation, a result by Benzerga~\cite[{Theorem~$2.5$}]{MR3473660} implies that $X$ has only finitely many real forms. This proves the claim. 
\end{proof}

The general statement of any smooth complex projective variety having at most countably many pairwise nonisomorphic real forms can be found in~\cite{labinet}. Therefore, rational affine surfaces are the first instance where we may have uncountably many pairwise nonisomorphic real forms.

\section{The moduli space} \label{section: moduli}

\subsection{Construction of moduli} \label{subsec: construction}

To describe the moduli which parametrises the real forms, we construct an auxiliary scheme. Let $C=\Spec(\C[\alpha^{\pm 1}, (\alpha-1)^{-1}])$ and set $\A^4_{C\times_{\C}C}=\Spec(\C[\alpha^{\pm 1}, (\alpha-1)^{-1}, \beta^{\pm 1}, (\beta -1)^{-1}][x,y,u,v])$. Consider the closed subscheme $\widetilde{\mathcal{S}}$ of $\A^4_{C\times_{\C}C}$ given by the ideal generated by the polynomials
\begin{align} \label{align: equations of affine surface}
		&yu - x(x-1)(x-\alpha), \nonumber \\
		&xv - u(u-1)(u-\beta), \\
		&yv - (x-1)(x-\alpha)(u-1)(u-\beta). \nonumber
\end{align}
Write $\widetilde{q}\colon \widetilde{\mathcal{S}} \rightarrow C\times_{\C} C$ for the restriction of $\mathrm{pr}_{C\times_{\C}C}\colon \A^4_{C\times_{\C}C} \rightarrow C\times_{\C} C$ to $\widetilde{S}$. For any closed point $(\alpha, \beta) \in (\C \smallsetminus \{0,1\})^2$ of $C\times_{\C} C$, the fibre $\widetilde{q}^{-1}(\alpha, \beta)$ is isomorphic to a surface in $\A^4_{\C}$ given by the equations in~\eqref{align: equations of affine surface} with the fixed $\alpha, \beta$; we denote both the fibre and this isomorphic surface by $S_{\alpha, \beta}$. 

\begin{remark} \label{remark: S_alpha,beta iso to S_beta,alpha}
	We have an isomorphism $S_{\alpha, \beta} \overset{\sim}{\rightarrow} S_{\beta, \alpha}, (x,y,u,v) \mapsto (u,v,x,y)$ induced by the involution on $C\times_{\C}C$ exchanging the two factors.
\end{remark}
	
Denote by $\mathcal{S} \subseteq \A^4_{C}$ the smooth affine relative surface obtained by restricting the relative surface $\widetilde{q} \colon \widetilde{\mathcal{S}} \rightarrow C \times_{\C} C$ over the diagonal. Consider the involution $\iota \colon \mathcal{S} \rightarrow \mathcal{S}, (\alpha, x,y,u,v) \mapsto (\alpha, u,v,x,y)$ and the restriction $(\sigma_C \times \sigma_{\A^4_{\C}})|_{\mathcal{S}}\colon \mathcal{S} \rightarrow \mathcal{S}, (\alpha, x,y,u,v) \mapsto (\overline{\alpha}, \overline{x}, \overline{y}, \overline{u}, \overline{v})$ of the product real structure on $\A^4_{C}=C \times_{\C} \A^4_{\C}$. Their composition defines a real structure $\sigma$ on $\mathcal{S}$, resulting in the morphism $q \colon (\mathcal{S}, \sigma) \rightarrow (C, \sigma_C)$. Any fibre of $q$ over real points of $(C, \sigma_C)$ is isomorphic to the smooth surface $S_{\alpha}=S_{\alpha, \alpha} \subseteq \A^4_{\C}$, with real structure $\sigma_{\alpha}$.
	
\begin{definition} \label{def: the main relative real variety}
	Define $p \colon \mathcal{X} \rightarrow \A^1_{\R} \smallsetminus \{0,1\}$ as the relative real affine variety corresponding to $q \colon (\mathcal{S}, \sigma) \rightarrow (C, \sigma_C)$ under the equivalence of categories in~\cite[{\RomanNumeralCaps{8}, Cor.~$7.7$}]{MR0217087}.
\end{definition}

The surface $S_{\alpha, \beta}$ is isomorphic to two other surfaces. Both involve the blow-up of the affine plane, where in the first instance, we blow up only points lying in $\A^2_{\C}$ and not infinitely close ones.

\subsection{First embeddings of the surfaces $S_{\alpha, \beta}$} \label{subsec: first}

In this section we describe one of the two ways of embedding $S_{\alpha, \beta}$ in the blow-up of the affine plane. 

Fix $\alpha, \beta \in \C \smallsetminus \{0,1\}$. Blow up the complex affine plane in the five points $(0,0)$, $(1,0)$, $(\alpha, 0)$, $(0,1)$ and $(0,\beta)$, and call the exceptional curves above them $E(0,0)$, $E(1,0)$, $E(\alpha,0)$, $E(0,1)$ and $E(0,\beta)$, respectively. Denote this blow-up by $\pi_{\alpha, \beta}\colon V_{\alpha, \beta} \rightarrow \A^2_{\C}$. Define $W_{\alpha, \beta}$ as the complex rational surface one obtains as the complement of the strict transforms of the lines $L_{x}=\{x=0\}$ and $L_{y}=\{y=0\}$ in the blow-up, that is $W_{\alpha, \beta} = V_{\alpha, \beta} \smallsetminus (L_x \cup L_y)$. Let $\iota_{\A^2_{\C}} \colon \A^2_{\C} \rightarrow \A^2_{\C}, (x,y) \mapsto (y,x)$ and let $\rho=\iota_{\A^2_{\C}} \circ \rho_{\A^2_{\C}}$, which defines a real structure on $\A^2_{\C}$. Whenever $\alpha \in \R\smallsetminus \{0,1\}$, since $L_x$ and $L_y$ are exchanged by $\rho$, we can endow $W_{\alpha}=W_{\alpha, \alpha}$ with the real structure $\rho_{\alpha}$ obtained as the restriction of the lift of $\rho$, see Remark~\ref{remark: blow-up}. 

Choose the standard open embedding $\A^2_{\C} \hookrightarrow \PP^2_{\C}, (x,y) \mapsto [x:y:1]$. Denote by $L_z=\{z=0\}$ the line at infinity. Then $W_{\alpha, \beta}=Y_{\alpha, \beta} \smallsetminus (L_x \cup L_y \cup L_z)$. Compare with Figure~\ref{figure: W_alpha,beta}. If $\alpha \in \R \smallsetminus \{0,1\}$, then the real structure $\iota_{\PP^2_{\C}} \circ \rho_{\PP^2_{\C}}$ lifts to a real structure $\widetilde{\rho}_{\alpha}$ on $Y_{\alpha}=Y_{\alpha, \alpha}$ and restricts to $\rho_{\alpha}$ on $W_{\alpha}$.  

\begin{figure}[h]
	\centering
	\scalebox{1}{\begin{tikzpicture}
			\draw (0.1,0) -- (2.9,0);
			\draw (0.2,0.2) -- (1.6,-2.3);
			\draw (2.8,0.2) -- (1.4,-2.3);
			\fill[green!70!blue] (1.5, -2.123) circle[radius=1.5pt];
			\fill[red] (1.1, -1.4) circle[radius=1.5pt];
			\fill[red] (0.7, -0.7) circle[radius=1.5pt];
			\fill[blue] (1.9, -1.4) circle[radius=1.5pt];
			\fill[blue] (2.3, -0.7) circle[radius=1.5pt];
			\node at (1.5,0.3) {\tiny $L_z$};
			\node at (0.2,0.4) {\tiny $L_{y}$};
			\node at (2.8,0.4) {\tiny $L_{x}$};
			\node at (0.04,-0.7) {\tiny $(1,0)$};
			\node at (0.4,-1.45) {\tiny $(\alpha,0)$};
			\node at (2.95,-0.7) {\tiny $(0,1)$};
			\node at (2.6,-1.45) {\tiny $(0,\beta)$};
			\node at (1.5,-2.55) {\tiny $(0,0)$};
			
			\node at (1.55, -3.2) {\small $\PP^2_{\C}$};
			
			\draw[line width=1pt] (6.3,0) -- (9.3,0);
			\draw[line width=1pt] (6.5,0.2) -- (6.5,-2.3);
			\draw[line width=1pt] (9.1,0.2) -- (9.1,-2.3);
			\draw[green!70!blue] (6.3,-2.1) -- (9.3,-2.1);
			\draw[red] (6.3,-1.4) -- (7.3,-1.4);
			\draw[red] (6.3,-0.7) -- (7.3,-0.7);
			\draw[blue] (8.3,-1.4) -- (9.3,-1.4);
			\draw[blue] (8.3,-0.7) -- (9.3,-0.7);
			
			\node at (7.8,0.3) {\tiny $L_z$};
			\node at (6.45,0.4) {\tiny $L_{y}$};
			\node at (9.1,0.4) {\tiny $L_{x}$};
			\node at (7.8,-2.4) {\tiny $E(0,0)$};
			\node at (5.7,-1.4) {\tiny $E(\alpha,0)$};
			\node at (5.75,-0.7) {\tiny $E(1,0)$};
			\node at (9.9,-1.4) {\tiny $E(0,\beta)$};
			\node at (9.9,-0.7) {\tiny $E(0,1)$};
			
			\node at (7.8, -3.2) {\small $Y_{\alpha, \beta}$};
			
			\draw[<-] (4,-1) -- (5,-1);			
	\end{tikzpicture}}
	\caption{Construction of $Y_{\alpha, \beta}$ and $W_{\alpha, \beta}$.} \label{figure: W_alpha,beta}
\end{figure}
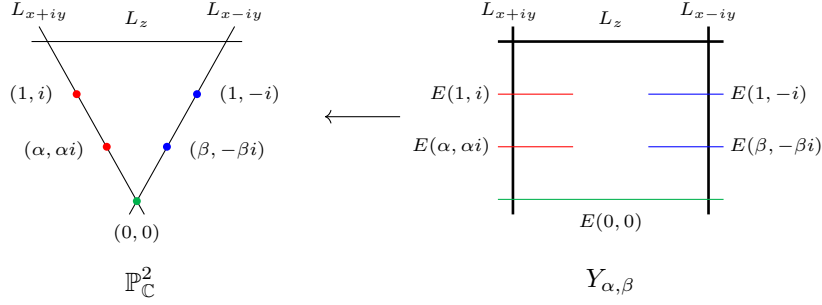

Note that on $Y_{\alpha, \beta}$, the curves $L_{y}$, $L_z$ and $L_{x}$ --- drawn with thick lines in Figure~\ref{figure: W_alpha,beta} --- have self-intersections $-2$, $1$ and $-2$, respectively. As they successively intersect in one point only, the chain $B = L_{y} \triangleright L_z \triangleright L_{x}$ is called a \textit{zigzag}, a term coined and studied extensively by Gizatullin and Danilov~\cite{MR0286791, MR0376701, MR0437545}, or more recently by Flenner, Kaliman and Zaidenberg~\cite{MR2426134}, or Blanc and Dubouloz~\cite{MR2817414, MR3340326}. As will be proved in Lemma~\ref{lemma: W_alpha,beta iso to S_alpha,beta} below, the surface $W_{\alpha, \beta}=Y_{\alpha, \beta} \smallsetminus B$ is affine and since it is also completable by a chain of smooth rational curves, $W_{\alpha, \beta}$ is called a \textit{Gizatullin surface}.

\subsection{Second  embeddings of the surfaces $S_{\alpha, \beta}$} \label{subsec: second}
The second incarnation of $S_{\alpha, \beta}$ is also given as the complement of a zigzag and the construction is very similar to the one in Section~\ref{subsec: first}, while also blowing up some points on exceptional curves. A key advantage of this construction will become apparent in the proof of Proposition~\ref{prop: all W_alpha,beta are isomorphic}. 

Fix $\alpha, \beta \in \C \smallsetminus \{0,1\}$. First, blow up the three points $(1,0), (\alpha,0), (0,0)$ in $\A^2_{\C}$, and call the exceptional curves above them $E(1,0)$, $E(\alpha, 0)$ and $E(0,0)$, respectively. Then, on $E(0,0)$, blow up the two points corresponding to the tangent directions $x-y=0$ and $x-\beta y=0$ and denote the exceptional curves by $A_1$ and $A_{\beta}$; compare with Figure~\ref{figure: X_alpha,beta}. Denote the resulting blow-up by $\gamma_{\alpha, \beta}\colon Z_{\alpha, \beta} \rightarrow \A^2_{\C}$. Call $X_{\alpha,\beta}$ the surface that results by removing $L_y$ and $E(0,0)$ from $Z_{\alpha, \beta}$.
	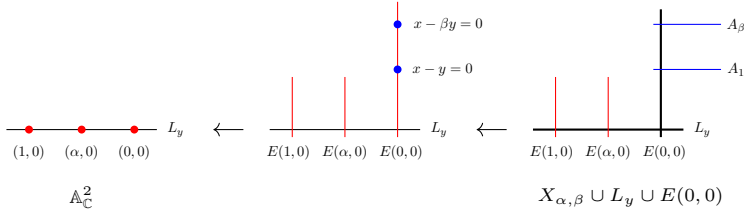
\begin{figure}[h]
		\centering
		\scalebox{1}{\begin{tikzpicture}
				\draw (0,0) -- (2,0);
				\fill[red] (0.3,0) circle[radius=1.5pt];
				\fill[red] (1,0) circle[radius=1.5pt];
				\fill[red] (1.7,0) circle[radius=1.5pt];
				\node at (2.25,0) {\scalebox{0.5}{$L_y$}};
				\node at (0.3,-0.3) {\scalebox{0.5}{$(1,0)$}};
				\node at (1,-0.3) {\scalebox{0.5}{$(\alpha,0)$}};
				\node at (1.7,-0.3) {\scalebox{0.5}{$(0,0)$}};
				
				
				\draw[<-] (2.75,0) -- (3.15,0);
				
				\draw (3.5,0) -- (5.5,0);
				\node at (5.75,0) {\scalebox{0.5}{$L_y$}};
				\draw[red] (5.2,-0.1) -- (5.2,1.7);
				\draw[red] (4.5,-0.1) -- (4.5,0.7);
				\draw[red] (3.8,-0.1) -- (3.8,0.7);
				\fill[blue] (5.2,0.8) circle[radius=1.5pt];
				\fill[blue] (5.2,1.4) circle[radius=1.5pt];
				\node at (5.8, 0.8) {\scalebox{0.5}{$x-y=0$}};
				\node at (5.875, 1.4) {\scalebox{0.5}{$x-\beta y=0$}};
				\node at (5.25,-0.3) {\scalebox{0.5}{$E(0,0)$}};
				\node at (4.5,-0.3) {\scalebox{0.5}{$E(\alpha,0)$}};
				\node at (3.75,-0.3) {\scalebox{0.5}{$E(1,0)$}};
				
				\draw[<-] (6.25,0) -- (6.65,0);
				
				\draw[thick] (7,0) -- (9,0);
				\node at (9.25,0) {\scalebox{0.5}{$L_y$}};
				\draw[thick] (8.7,-0.1) -- (8.7,1.6);
				\draw[red] (8,-0.1) -- (8,0.7);
				\draw[red] (7.3,-0.1) -- (7.3,0.7);
				\draw[blue] (8.6,0.8) -- (9.5,0.8);
				\draw[blue] (8.6,1.4) -- (9.5,1.4);
				\node at (8.75,-0.3) {\scalebox{0.5}{$E(0,0)$}};
				\node at (8,-0.3) {\scalebox{0.5}{$E(\alpha,0)$}};
				\node at (7.25,-0.3) {\scalebox{0.5}{$E(1,0)$}};
				\node at (9.7,0.8) {\scalebox{0.5}{$A_1$}};
				\node at (9.7,1.375) {\scalebox{0.5}{$A_{\beta}$}};
		\end{tikzpicture}}
		\caption{Construction of the affine rational surface $X_{\alpha, \beta}$.} \label{figure: X_alpha,beta}
	\end{figure}

\subsection{Isomorphisms with the closed fibres}

We now prove that the surfaces $W_{\alpha, \beta}$ and $X_{\alpha, \beta}$ from Sections~\ref{subsec: first} and~\ref{subsec: second} are isomorphic to the closed fibres $S_{\alpha, \beta}$ defined in~\ref{subsec: construction}. We keep the same notation for $\pi_{\alpha, \beta}\colon V_{\alpha, \beta} \rightarrow \A^2_{\C}$ and $\gamma_{\alpha, \beta}\colon Z_{\alpha, \beta} \rightarrow \A^2_{\C}$ as in Sections~\ref{subsec: first} and~\ref{subsec: second}.

\begin{lemma}\label{lemma: W_alpha,beta iso to S_alpha,beta}
	~
	\begin{enumerate}[leftmargin=*]
		\item For all $\alpha, \beta \in \C \smallsetminus \{0,1\}$, the morphism $\nu_{\alpha, \beta}=\mathrm{pr}_{x,u}\colon S_{\alpha, \beta} \rightarrow \A^2_{\C}, (x,y,u,v) \mapsto (x, u)$ is birational and factors as $\nu_{\alpha, \beta}=\pi_{\alpha, \beta}\circ j_{\alpha, \beta}$, where $j_{\alpha, \beta} \colon S_{\alpha, \beta} \hookrightarrow V_{\alpha, \beta}$ is an open embedding with image $W_{\alpha, \beta}$. \label{item: S_alpha,beta isomorphic to W_alpha,beta}
		\item For a fixed $\alpha \in \R \smallsetminus \{0,1\}$, the induced isomorphism $j_{\alpha}=j_{\alpha, \alpha} \colon (S_{\alpha}, \sigma_{\alpha}) \overset{\sim}{\rightarrow} (W_{\alpha}, \rho_{\alpha})$ is real. \label{item: real iso between S_alpha and W_alpha}
		\item For all $\alpha, \beta \in \C \smallsetminus \{0,1\}$, the morphism $\eta_{\alpha, \beta}=\mathrm{pr}_{x,y}\colon S_{\alpha, \beta} \rightarrow \A^2_{\C}, (x,y,u,v) \mapsto (x, \tfrac{1}{\alpha}y)$ is birational and factors as $\nu_{\alpha, \beta}=\gamma_{\alpha, \beta}\circ i_{\alpha, \beta}$, where $i_{\alpha, \beta} \colon S_{\alpha, \beta} \hookrightarrow Z_{\alpha, \beta}$ is an open embedding with image $X_{\alpha, \beta}$. \label{item: S_alpha,beta isomorphic to X_alpha,beta}
	\end{enumerate}
\end{lemma}
\begin{proof}
	We prove~\ref{item: S_alpha,beta isomorphic to W_alpha,beta}: The restriction of $\nu_{\alpha, \beta}$ to $S_{\alpha, \beta} \smallsetminus \{xu=0\}$ gives an isomorphism with $\A^2_{\C} \smallsetminus (L_{x} \cup L_{y})$: any point outside of $L_{x}$ and $L_{y}$ comes from a point $(x,y,u,v)$ with $x,u\neq 0$. For such points, we find using~\eqref{align: equations of affine surface} that $y=\tfrac{x(x-1)(x-\alpha)}{u}$ and $v=\tfrac{u(u-1)(u-\beta)}{x}$, and for such $y$ and $v$, we find
	\begin{align*}
		yv&=\frac{x(x-1)(x-\alpha)u(u-1)(u-\beta)}{ux} \\
		&= (x-1)(x-\alpha)(u-1)(u-\beta).
	\end{align*}
	Thus, for $x, u \neq 0$, we can find a unique point $(x,y,u,v) \in S_{\alpha, \beta}$, which shows the isomorphism $S_{\alpha, \beta} \smallsetminus \{xu=0\} \overset{\sim}{\rightarrow} \A^2_{\C}\smallsetminus (L_{x} \cup L_{y})$ and thus the birationality of $\nu_{\alpha, \beta}$.
	
	Now, $\nu_{\alpha, \beta}$ maps the principal divisor
	\begin{align*}
		\{x=0\} &\cong \Spec(\C[y,u,v]/(yu,u(u-1)(u-\beta), yv-\alpha(u-1)(u-\beta)))\\
		&\cong \{x=y=u-1=0\} \sqcup \{x=y=u-\beta=0\}~\\
		&\phantom{\cong} \phantom{\cong} \sqcup \{x=u=yv-\alpha\beta=0\}
	\end{align*}
	on $S_{\alpha, \beta}$ onto $(0,1)$, $(0,\beta)$ and $(0,0)$, compare with Equations~\eqref{align: equations of affine surface}. Similarly, 
	\begin{align*}
		\{u=0\} &\cong \Spec(\C[x,y,v]/(x(x-1)(x-\alpha),xv, yv-\beta(x-1)(x-\alpha)))\\
		&\cong \{u=x=yv-\alpha\beta\} \sqcup \{u=v=x-1=0\}~\\
		&\phantom{\cong} \phantom{\cong} \sqcup \{u=v=x-\alpha=0\}
	\end{align*}
	is sent onto the three points $(0,0)$, $(1,0)$ and $(\alpha,0)$. Scheme theoretically, the fibres of $\nu_{\alpha, \beta}$ over these five points are of codimension~$1$ in $S_{\alpha, \beta}$. Since $S_{\alpha, \beta}$ is smooth, these five fibres are Cartier divisors on $S_{\alpha, \beta}$. Therefore, the inverse image under $\nu_{\alpha, \beta}$ of the ideal sheaf of $\{(0,0), (1,0), (\alpha,0), (0,1), (0,\beta)\}$ is an invertible sheaf on $S_{\alpha, \beta}$, and we can apply the universal property of the blow-up, see for example~\cite[\RomanNumeralCaps{2}.7.14]{MR0463157}. Thus, there exists a unique morphism $j_{\alpha, \beta}\colon S_{\alpha, \beta} \rightarrow V_{\alpha, \beta}$ such that $\nu_{\alpha, \beta}=\pi_{\alpha, \beta}\circ j_{\alpha, \beta}$.
	
	It remains to be shown that $j_{\alpha, \beta}$ is an open immersion. Consider first the point $(0,0)$. Since $V_{\alpha, \beta}$ is the blow-up of $\A^2_{\C}$ at five points, including $(0,0)$, we can write it locally in a neighbourhood of $(0,0)$ as the subvariety $\{xs-ut=0\} \subseteq \A^2_{\C} \times_{\C} \PP^1_{\C}$. On this neighbourhood, $j_{\alpha, \beta}$ is given as $(x,y,u,v) \mapsto ((x,u), [x:u])$. Using the equations defining $S_{\alpha, \beta}$ in~\eqref{align: equations of affine surface}, we find on this chart
	\[[x:u]=[u(u-1)(u-\beta):v]=[y:(x-1)(x-\alpha)].\]
	Thus, the restriction of $j_{\alpha, \beta}$ to $\nu_{\alpha, \beta}^{-1}(0,0)=\{x=u=yv-\alpha\beta\}$ is the open immersion $\Spec(\C[y^{\pm 1}]) \hookrightarrow \{((0,0), [y:\alpha])\thinspace | \thinspace yv-\alpha\beta=0\}$ of $\A^1_{\C} \smallsetminus \{0\}$ in the exceptional divisor $E(0,0)$. It is thus the complement of $[0:1]$ and $[1:0]$ in $E(0,0)$, which is the intersection of $E(0,0)$ with the strict transform of $\{xu=0\}$. Therefore, $j_{\alpha, \beta}$ does not contract $\nu_{\alpha, \beta}^{-1}(0,0)$. Proceeding analogously for the remaining points $(1,0), (\alpha,0), (0,1), (0,\beta)$, one deduces that $j_{\alpha, \beta}$ does not contract any curve in $S_{\alpha, \beta}$.
	
	This implies that $j_{\alpha, \beta}\colon S_{\alpha, \beta} \rightarrow V_{\alpha, \beta}$ is a quasi-finite birational morphism with normal target. Thus, by the version of Zariski's main theorem in~\cite[{Theorem~$8.12.6$}]{MR217086}, indeed, $j_{\alpha, \beta}$ is an open immersion. Since $\nu_{\alpha, \beta}=\pi_{\alpha, \beta}\circ j_{\alpha, \beta}$, the image of $j_{\alpha, \beta}$ is equal to $W_{\alpha, \beta}$. This proves~\ref{item: S_alpha,beta isomorphic to W_alpha,beta}. Item~\ref{item: real iso between S_alpha and W_alpha} follows by construction of $j_{\alpha}=j_{\alpha, \alpha}$.
	
	As for~\ref{item: S_alpha,beta isomorphic to X_alpha,beta}, we essentially repeat the arguments made to prove~\ref{item: S_alpha,beta isomorphic to W_alpha,beta}, with the caveat that in order to obtain $X_{\alpha, \beta}$ we blow up two points which lie on the exceptional curve $E(0,0)$. The restriction of $\eta_{\alpha, \beta}$ induces an isomorphism between $S_{\alpha, \beta} \smallsetminus \{y=0\}$ and $\A^2_{\C} \smallsetminus \{y=0\}$. This is again seen using the equations in~\eqref{align: equations of affine surface}: For any $y\neq 0$ and a fixed $x \in \C$, we find unique $u=\tfrac{x(x-1)(x-\alpha)}{y}$ and $v=\tfrac{(x-1)(x-\alpha)(u-1)(u-\beta)}{y}$, which satisfy $xv=\tfrac{x(x-1)(x-\alpha)(u-1)(u-\beta)}{y}=u(u-1)(u-\beta)$. This shows that $\eta_{\alpha, \beta}$ is birational.
	
	Consider the principal divisor 
	\begin{align*} 
		\{y=0\} &\cong \{y=x=u-1=0\} \sqcup \{y=x=u-\beta=0\} \\
		& \phantom{\cong} \phantom{\cong} \sqcup \{y=x-1=v-u(u-1)(u-\beta)=0\} \\
		& \phantom{\cong} \phantom{\cong} \sqcup \{y=x-\alpha=\alpha v-u(u-1)(u-\beta)=0\}
	\end{align*}
	on $S_{\alpha, \beta}$. These components are mapped to $(0,0)$ (the first two), $(1,0)$ and $(\alpha,0)$. Scheme theoretically, the fibres above these three points are of codimension~$1$ and must be Cartier divisors by the smoothness of $S_{\alpha, \beta}$. Note that $\eta_{\alpha, \beta}^{-1}(0,0)$ consists of two connected components, and $\eta_{\alpha, \beta}^{-1}(1,0)$ and $\eta_{\alpha, \beta}^{-1}(\alpha, 0)$ are both irreducible. As before, $\eta_{\alpha, \beta}$ must have a unique lift $\widetilde{\eta}_{\alpha, \beta} \colon S_{\alpha, \beta} \rightarrow \widetilde{Z}_{\alpha, \beta}$ to the blow-up $\widetilde{\gamma}_{\alpha, \beta} \colon \widetilde{Z}_{\alpha, \beta} \rightarrow \A^2_{\C}$ of $\A^2_{\C}$ in $(0,0)$, $(1,0)$ and $(\alpha, 0)$. Call $p_1, p_{\beta} \in E(0,0) \smallsetminus L_y \subseteq \widetilde{Z}_{\alpha, \beta}$ the two points on which the connected components of $\eta_{\alpha, \beta}^{-1}(0,0)$ are contracted by $\widetilde{\eta}_{\alpha, \beta}$. Locally around $(0,0)$ the surface $\widetilde{Z}_{\alpha, \beta}$ is given by $\{ys-xt=0\} \subseteq \A^2_{\C} \times_{\C} \PP^1_{\C}$. Thus the map $\widetilde{\eta}_{\alpha, \beta}$ can be written on this neighbourhood as $(x,y,u,v) \mapsto ((x,\tfrac{1}{\alpha}y), [x:\tfrac{1}{\alpha}y])$ and using~\eqref{align: equations of affine surface}, we deduce
	\[[x:y]=[u:\tfrac{1}{\alpha}(x-1)(x-\alpha)].\]
	Thus, $p_1=((0,0),[1:1])$ and $p_\beta=((0,0),[\beta:1])$. Repeating the Cartierness argument, the inverse image of the ideal sheaf of the two points is invertible on $\widetilde{Z}_{\alpha, \beta}$. This way, we obtain a unique lift $i_{\alpha, \beta}\colon S_{\alpha, \beta} \rightarrow Z_{\alpha, \beta}$ such that $\eta_{\alpha, \beta}=\gamma_{\alpha, \beta} \circ i_{\alpha, \beta}$. Here, $\gamma_{\alpha, \beta}\colon Z_{\alpha, \beta} \rightarrow \A^2_{\C}$ is the composition of the blow-up $\widetilde{\pi}_{\alpha, \beta} \colon Z_{\alpha, \beta} \rightarrow \widetilde{Z}_{\alpha, \beta}$ of $p_1$ and $p_{\beta}$ with $\widetilde{\gamma}_{\alpha, \beta}\colon \widetilde{Z}_{\alpha, \beta} \rightarrow \A^2_{\C}$, the blow-up of $(0,0)$, $(1, 0)$ and $(\alpha, 0)$.
	
	In order to once more apply the version of Zariski's main theorem as in~\cite[{Theorem~$8.12.6$}]{MR217086}, we need to see that $i_{\alpha, \beta}$ restricts to open immersions into $\widetilde{\gamma}_{\alpha, \beta}^{-1}(1,0)$, $\widetilde{\gamma}_{\alpha, \beta}^{-1}(\alpha,0)$, $\widetilde{\pi}_{\alpha, \beta}^{-1}(p_1)$ and $\widetilde{\pi}_{\alpha, \beta}^{-1}(p_\beta)$. In the case of $\widetilde{\gamma}_{\alpha, \beta}^{-1}(1,0)$ and $\widetilde{\gamma}_{\alpha, \beta}^{-1}(\alpha,0)$, the argument is analogous to before when we discussed the restriction of $j_{\alpha, \beta}$ to $\nu_{\alpha, \beta}^{-1}(0,0)$. What remains to be analysed is the restriction of $i_{\alpha, \beta}$ to $\widetilde{\pi}_{\alpha, \beta}^{-1}(p_1)$ and $\widetilde{\pi}_{\alpha, \beta}^{-1}(p_\beta)$. We do the calculations for $p_1$; the ones for $p_{\beta}$ proceed analogously. 
	
	Locally around $(0,0)$, the surface $\widetilde{Z}_{\alpha, \beta}$ is given by $\{ys-xt=0\} \subseteq \A^2_{\C} \times_{\C}\PP^1_{\C}$. Blowing up the point $p_1=((0,0),[1,1])$, on the open chart $t=1$ --- which coincides with $x=0$ on $S_{\alpha, \beta}$ --- the surface $Z_{\alpha, \beta}$ is then locally given by $\{p(s-1)-qy=0\} \subseteq \A^2_{\C} \times_{\C} \PP^1_{\C}=\Spec(\C[y,s]) \times_{\C} \Proj(\C[p,q])$. Thus, on that neighbourhood, $\eta_{\alpha, \beta}$ can be written as $(0,y,u,v) \mapsto ((0,\tfrac{1}{\alpha}y), [u:1], [u-1:\tfrac{1}{\alpha}y])$. Using the equations in~\eqref{align: equations of affine surface} with $x=0$, we find
	\[[u-1:\tfrac{1}{\alpha}y]=[\alpha(u-1)(u-\beta): y(u-\beta)]=[v:u-\beta].\]
	Thus, the restriction of $i_{\alpha, \beta}$ to $\pi_{\alpha, \beta}^{-1}(p_1)=\{y=x=u-1=0\}$ is the open immersion $\Spec(\C[v]) \hookrightarrow \{((0,0), [1:1], [v:1-\beta])\}$ of $\A^1_{\C}$ in the exceptional divisor $A_1$. It is thus the complement of $[1:0]$ in $A_1$, which is the intersection of $A_1$ with $E(0,0)$.
	
	As before, we deduce that $i_{\alpha, \beta} \colon S_{\alpha, \beta} \rightarrow Z_{\alpha, \beta}$ is a quasi-finite birational morphism, and thus an open immersion. This proves~\ref{item: S_alpha,beta isomorphic to X_alpha,beta} and finishes the proof.
\end{proof}

We will work with both isomorphic surfaces $W_{\alpha, \beta}$ and $X_{\alpha, \beta}$ from Sections~\ref{subsec: first} and~\ref{subsec: second}, since we need the $X_{\alpha, \beta}$'s to allow for a short proof in Section~\ref{section: closed fibres are pairwise isomorphic}, whereas in Sections~\ref{section: birational diffeomorphisms} and~\ref{section: nonisomorphic real forms}, formulating the proof with the help of $(W_{\alpha},\rho_{\alpha})$ makes the exposition clearer. 

\section{Closed complex fibres are pairwise isomorphic} \label{section: closed fibres are pairwise isomorphic}

To show that the closed fibres of $p_{\C}\colon \mathcal{X}_{\C} \rightarrow \A^1_{\C}\smallsetminus \{0,1\}$ are pairwise isomorphic, we prove a stronger result. For this, recall the morphism $\widetilde{q}\colon \widetilde{S} \rightarrow C \times_{\C} C$ defined in Section~\ref{subsec: construction}. It has fibres over closed points $(\alpha, \beta)$ with $\alpha, \beta \in \C \smallsetminus \{0,1\}$ isomorphic to the surface $S_{\alpha, \beta}$ given in $\A^4_{\C}$ by the equations in~\eqref{align: equations of affine surface}. When we restrict it over the diagonal, we obtain $q \colon \mathcal{S} \rightarrow C$, which is isomorphic to $p_{\C} \colon \mathcal{X} \rightarrow \A^1_{\C} \smallsetminus \{0,1\}$.

\begin{proposition} \label{prop: all W_alpha,beta are isomorphic}
	Any two fibres of $\widetilde{q}\colon \widetilde{S} \rightarrow C \times_{\C} C$ over two closed points of $C\times_{\C} C$ are isomorphic, that is, for any $\alpha_1, \beta_1, \alpha_2, \beta_2 \in \C \smallsetminus \{0,1\}$ we have $S_{\alpha_1, \beta_1} \cong S_{\alpha_2, \beta_2}$. In particular, the closed fibres of $p_{\C}\colon \mathcal{X}_{\C} \rightarrow \A^1_{\C} \smallsetminus \{0,1\}$ are pairwise isomorphic.
\end{proposition}
\begin{proof}
	We first show that for any $\alpha, \beta \in \C \smallsetminus \{0,1\}$, the surfaces $X_{\alpha, \beta}$ and $X_{\alpha, \alpha}$ defined in Section~\ref{subsec: second} are isomorphic. It suffices to find a linear automorphism of $\A^2_{\C}$ which restricts to an automorphism of $L_y$ sending $(1,0)$, $(\alpha, 0)$ and $(0,0)$ to $(1,0)$, $(\alpha, 0)$ and $(0,0)$, respectively, and which sends the tangent directions $x-y=0$ and $x-\beta y=0$ at $(0,0)$ to the tangent directions $x-y=0$ and $x-\alpha y=0$ at $(0,0)$, respectively. These requirements are fulfilled by $(x,y) \mapsto (x+\tfrac{\alpha - \beta}{1-\alpha}y, \tfrac{1-\beta}{1-\alpha}y)$. This shows $X_{\alpha, \beta} \cong X_{\alpha, \alpha}$. From this, using both Lemma~\ref{lemma: W_alpha,beta iso to S_alpha,beta}~\ref{item: S_alpha,beta isomorphic to X_alpha,beta} and Remark~\ref{remark: S_alpha,beta iso to S_beta,alpha} repeatedly, $S_{\alpha_1, \beta_1} \cong X_{\alpha_1, \beta_1} \cong X_{\alpha_1, \alpha_2} \cong S_{\alpha_1, \alpha_2} \cong S_{\alpha_2, \alpha_1} \cong X_{\alpha_2, \alpha_1} \cong X_{\alpha_2, \beta_2} \cong S_{\alpha_2, \beta_2}$. This proves the first claim.
	
	Since $p_{\C}$ is just $\widetilde{q}$ restricted over the diagonal, the second claim follows immediately from the first, concluding the proof.
\end{proof}

Proposition~\ref{prop: all W_alpha,beta are isomorphic} implies Statement~\ref{item: main theorem, 1} of Theorem~\ref{thm: new main theorem}.

\section{Birational diffeomorphisms} \label{section: birational diffeomorphisms}

Not only are the $S_{\alpha}$ isomorphic over $\C$; whenever $\alpha \in \R \smallsetminus \{0,1\}$, the real loci of the $(S_{\alpha}, \rho_{\alpha})$'s are all birationally diffeomorphic, and we can describe the real locus explicitly.

\begin{proposition} \label{prop: birationally diffeomorphic}
	For any $\alpha \in \R\smallsetminus \{0, 1\}$, the real surface $(S_{\alpha},\sigma_{\alpha})$ is birationally diffeomorphic to the affine plane blown up in the origin. In particular, the fibres over $\R$-rational closed points of $p\colon \mathcal{X}\rightarrow \A^1_{\R} \smallsetminus \{0,1\}$ are pairwise birationally diffeomorphic smooth real rational surfaces.
\end{proposition}
\begin{proof}
	By Lemma~\ref{lemma: W_alpha,beta iso to S_alpha,beta}~\ref{item: S_alpha,beta isomorphic to W_alpha,beta} and~\ref{item: real iso between S_alpha and W_alpha}, we can consider $(W_{\alpha}, \rho_{\alpha})$ instead of $(S_{\alpha}, \sigma_{\alpha})$. The real points of $(W_{\alpha},\rho_{\alpha})$ are either coming from points of $\A^2_{\C}$ invariant under $(x,y)\mapsto (\overline{y},\overline{x})$ or lying on an exceptional curve which is the blow-up of a point in $\A^2_{\C}$ invariant under $(x,y)\mapsto (\overline{y},\overline{x})$. Since the only point on $L_{x}\cup L_{y}$ which is invariant under this action on $\A^2_{\C}$ is $(0,0)$, the points of $W_{\alpha}$ on $E(1,0)$, $E(\alpha, 0)$, $E(0,1)$ and $E(0, \alpha)$ cannot be part of the real locus, compare with Figure~\ref{figure: W_alpha,beta}. Thus, the real locus of $(W_{\alpha}, \rho_{\alpha})$ is birationally diffeomorphic to the blow-up of $\A^2_{\R}$ in the origin. This proves the first claim.
	
	As the fibres of $p$ over $\R$-rational closed points are isomorphic to the $(W_{\alpha}, \rho_{\alpha})$ with $\alpha \in \R\smallsetminus \{0,1\}$, they are all pairwise birationally diffeomorphic. This finishes the proof.
\end{proof}

Proposition~\ref{prop: birationally diffeomorphic} implies Statement~\ref{item: main theorem, 2} of Theorem~\ref{thm: new main theorem}. Also note that the affine plane blown up in a point is homeomorphic to the open Moebius band, which further describes the real locus of $(W_{\alpha}, \rho_{\alpha})$ topologically.

\section{Nonisomorphic real forms} \label{section: nonisomorphic real forms}

This sections is dedicated to proving that fibres of $p\colon \mathcal{X} \rightarrow \A^1_{\R} \smallsetminus \{0,1\}$ over $\R$-rational closed points $\alpha, \beta$ are isomorphic if and only if $\alpha=\beta$ or $\alpha\beta=1$. We first prove an auxiliary result on the curves of negative self-intersection on $Y_{\alpha}$, the surfaces one obtains by completing $W_{\alpha}$ with the zigzag in Section~\ref{subsec: first}.

\begin{lemma} \label{lemma: all curves of negative self-intersection on Y_alpha}
	Fix $\alpha \in \C\smallsetminus \{0,1\}$. The curves of negative self-intersection on $Y_{\alpha}$ are precisely the exceptional curves $E(0,0)$, $E(1,0)$, $E(\alpha, 0)$, $E(0,1)$, $E(0,\alpha)$ and the strict transforms of the lines $L_{x}$, $L_{y}$, $L_{x+y-z}$, $L_{x+y-\alpha z}$, $L_{x+\alpha y-\alpha z}$ and $L_{\alpha x +y -\alpha z}$.
\end{lemma}
\begin{proof}
	Consider an irreducible curve $C \subseteq \PP^2_{\C}$ of degree~$d$ not equal to $L_{x}$ or $L_{y}$ whose strict transform $\widetilde{C}$ on $Y_{\alpha}$ has negative self-intersection. Denote the multiplicities of $C$ at $[0:0:1]$, $[1:0:1]$, $[\alpha :0:1]$, $[0:1:1]$ and $[0:\alpha :1]$ by $m_1$, $m_2$, $m_3$, $m_4$ and $m_5$, respectively. Since the curve $\widetilde{C}$ has negative self-intersection, we find $d^2-\sum_{i=1}^{5}m_i^2\leq -1$. The geometric genus of $\widetilde{C}$ satisfies
	\begin{align*}
		0\leq 2g &=(d-1)(d-2)-\sum_{i=1}^{5}m_i(m_i-1) \\
		&=d^2-\sum_{i=1}^{5}m_i^2+1 -3d+1+\sum_{i=1}^{5}m_i,
	\end{align*}
	which combined with the negative self-intersection results in $0\leq \sum_{i=1}^{5}m_i - 3d +1$. Taking the intersection of $C$ with any irreducible conic passing through $[1:0:1]$, $[\alpha :0:1]$, $[0:1:1]$ and $[0:\alpha :1]$, we obtain $2d-m_2-m_3-m_4-m_5\geq 0$, since any such conic cannot pass through $(0,0)$. We therefore find $d-1 \leq m_1$, which implies $m_1=d-1$. Taking the intersection of the curve $C$ with  $L_{x}$ and $L_{y}$ implies $d-m_1-m_2-m_3\geq0$ and $d-m_1-m_4-m_5\geq0$, hence $m_2+m_3\leq 1$ and $m_4+m_5\leq 1$. This shows $\sum_{i=2}^{5}m_i^2\leq 2$, and thus $-1\geq d^2-(d-1)^2-\sum_{i=2}^{5}m_i^2\geq 2d-1-2$. Thus, $d=1$, $m_1=0$ and $(m_2, m_3, m_4, m_5) \in \{(1,0,1,0), (1,0,0,1), (0,1,0,1), (0,1,1,0)\}$, which give precisely the lines as stated in the lemma. This proves the claim.
\end{proof}

Using Lemma~\ref{lemma: all curves of negative self-intersection on Y_alpha}, we can show that if two real surfaces $(W_{\alpha}, \rho_{\alpha})$ and $(W_{\beta}, \rho_{\beta})$ are isomorphic, then this isomorphism is in fact the lift of a  linear automorphism on $(\A^2_{\C}, \rho)$. Here, $\rho$ is the real structure defined in~\ref{subsec: first}.

\begin{lemma} \label{lemma: all isos of W_alpha are linear}
	Fix $\alpha, \beta \in \R\smallsetminus \{0,1\}$. Then any isomorphism $\varphi \colon (W_{\alpha}, \rho_{\alpha}) \overset{\sim}{\rightarrow} (W_{\beta}, \rho_{\beta})$ arises from a linear automorphism of $(\A^2_{\C}, \rho)$ sending $V(xy)$ to itself and the sets of points $\{(1,0), (\alpha,0)\}$ and $\{(0,1), (0, \alpha)\}$ to $\{(1,0), (\beta,0)\}$ and $\{(0,1), (0,\beta)\}$ or vice versa.
\end{lemma}
\begin{proof}
	Consider the compactifications $(Y_{\alpha}, \widetilde{\rho}_{\alpha})$ and $(Y_{\beta}, \widetilde{\rho}_{\beta})$ of $(W_{\alpha}, \rho_{\alpha})$ and $(W_{\beta}, \rho_{\beta})$ (see~Section~\ref{subsec: first}). Here,  $\widetilde{\rho}_{\alpha}$ is the real structure which arises from the real structure $[x:y:z] \mapsto [\overline{y}:\overline{x}:\overline{z}]$ on $\PP^2_{\C}$, and which restricts to $\rho_{\alpha}$ when restricting to $W_{\alpha}$; similarly for $\widetilde{\rho}_{\beta}$. Consider an isomorphism $\varphi \colon (W_{\alpha}, \rho_{\alpha}) \overset{\sim}{\rightarrow} (W_{\beta}, \rho_{\beta})$, which induces a birational map $\widetilde{\varphi} \colon (Y_{\alpha}, \widetilde{\rho}_{\alpha}) \overset{\sim}{\dashrightarrow} (Y_{\beta}, \widetilde{\rho}_{\beta})$. We call a base point of a birational map proper if it is a point of the domain, and not lying on an exceptional curve after some other blow-ups.
	
	If $\widetilde{\varphi}$ is not a morphism, then at least one proper base point of $\widetilde{\varphi}$ is the image of an irreducible curve of nonnegative self-intersection in $Y_{\beta}\smallsetminus W_{\beta}$ by $\widetilde{\varphi}^{-1}$. Since $Y_{\beta}\smallsetminus W_{\beta}$ is the union of $L_{x}$, $L_{y}$ and $L_z$ having self-intersections $-2$, $-2$ and $1$, this curve contracted by $\widetilde{\varphi}^{-1}$ must be $L_z$. As $\widetilde{\varphi}$ is real with respect to the lifts $\widetilde{\rho}_{\alpha}, \widetilde{\rho}_{\beta}$, this point can only lie on $L_z\smallsetminus (L_{x}\cup L_{y})$ in $Y_{\alpha}$. By an analogous analysis with the roles of $\widetilde{\varphi}$ and $\widetilde{\varphi}^{-1}$ reversed, we find that $L_z$ is contracted by $\widetilde{\varphi}$. But then, after a sequence of blow-ups, since $L_z$ was contracted and no base point lies on $L_{x}$ or $L_{y}$, the strict transforms of $L_{x}$ and $L_{y}$ become $(-1)$-curves, which have to be contracted again. This is impossible, since we cannot contract one without contracting the other, as they are exchanged by the real structures $\widetilde{\rho}_{\alpha}, \widetilde{\rho}_{\beta}$.
	
	Thus, $\widetilde{\varphi}$ cannot have any base points at all, and is therefore an isomorphism between the two projective surfaces $(Y_{\alpha}, \widetilde{\rho}_{\alpha})$ and $(Y_{\beta}, \widetilde{\rho}_{\beta})$ which sends $L_z$ to $L_z$. By Lemma~\ref{lemma: all curves of negative self-intersection on Y_alpha}, the irreducible curves of negative self-intersection on $Y_{\alpha}$ are precisely the exceptional curves $E(0,0)$, $E(1,0)$, $E(\alpha, 0)$, $E(0,1)$, $E(0,\alpha)$, each with self-intersection $-1$, and the strict transforms of the lines $L_{x}$, $L_{y}$, $L_{x+y-z}$, $L_{x+y-\alpha z}$, $L_{x+\alpha y -\alpha z}$ and $L_{\alpha x+y-\alpha z}$, with self-intersection $-2$, $-2$, $-1$, $-1$, $-1$, and $-1$, respectively. Thus, $L_{x}$ and $L_{y}$ are either sent to themselves or exchanged; and by considering their intersections with the other curves, $E(0,0)$ has to be sent to itself, and either $\{E(1,0), E(\alpha, 0)\}$ is sent to $\{E(1,0), E(\beta, 0)\}$, and $\{E(0,1), E(0,\alpha)\}$ is sent to $\{E(0,1), E(0,\beta)\}$, or vice versa. But this implies that $\widetilde{\varphi}$ is the lift of an automorphism of $(\PP^2_{\C}, [x:y:z] \mapsto [\overline{y}:\overline{x}:\overline{z}])$.
	
	We can therefore deduce that the isomorphism $\varphi$ arises from a linear automorphism of $(\A^2_{\C}, \rho)$ which sends $L_z$ to $L_z$, $V(xy)$ to itself and the points $\{(1,0), (\alpha,0)\}$ to $\{(1,0), (\beta,0)\}$, and $\{(0,1), (0,\alpha)\}$ to $\{(0,1), (0,\beta)\}$, or vice versa, as claimed.
\end{proof}

We now have all tools needed to prove the content of the last statement of Theorem~\ref{thm: new main theorem}, which we rephrase in the following proposition.

\begin{proposition} \label{prop: third assertion of main theorem}
	Any two fibres $p^{-1}(\alpha)\cong(W_{\alpha}, \rho_{\alpha})$, $p^{-1}(\beta)\cong(W_{\beta}, \rho_{\beta})$ for $\alpha, \beta \in \R \smallsetminus \{0,1\}$ of $p\colon \mathcal{X} \rightarrow \A^1_{\R} \smallsetminus \{0,1\}$ are isomorphic if and only if $\alpha\beta=1$ or $\alpha=\beta$.
\end{proposition}
\begin{proof}
	By Lemma~\ref{lemma: all isos of W_alpha are linear}, any $\R$-isomorphism $(W_{\alpha}, \rho_{\alpha}) \overset{\sim}{\rightarrow} (W_{\beta}, \rho_{\beta})$ arises from a linear automorphism of $(\A^2_{\C}, (x,y)\mapsto (\overline{y},\overline{x}))$ sending $V(xy)=L_x\cup L_y$ to itself and the sets of points $\{(1,0), (\alpha,0)\}$ and $\{(0,1), (0,\alpha)\}$ to $\{(1,0), (\beta,0)\}$ and $\{(0,1), (0,\beta)\}$, or vice versa. 
	
	Any linear automorphism of $(\A^2_{\C}, \rho)$ is of the form $(x,y) \mapsto (ax+by, cx+dy)$ with $\left(\begin{smallmatrix}
		a & b \\ c&d
	\end{smallmatrix}\right) \in \GL_2(\C)$ and $\overline{a}=d$, $\overline{b}=c$. Since the automorphism under consideration has to send $V(xy)=L_x\cup L_y$ to itself, we find it can only be one of the following: $\psi_1\colon (x,y) \mapsto (ax, \overline{a}y)$ or $\psi_2\colon (x,y) \mapsto (by, \overline{b}x)$. Under $\psi_1$, the set $\{(1,0), (\alpha,0)\}$ can only be sent to $\{(1,0), (\beta,0)\}$. Then either $\alpha=\beta$ and $a=1$ or $\alpha\beta=1$ and $a=\beta$. The same argument with $\psi_2$ sending $\{(1,0), (\alpha,0)\}$ to $\{(0,1), (0,\beta)\}$ results in $\alpha=\beta$ and $b=1$, or $\alpha\beta=1$ and $b=\beta$.
	
	Conversely, if $\alpha\beta=1$, set $\varphi \in \Aut(\A^2_{\C}, \rho)$ to be defined by $(x,y)\mapsto (\beta x, \beta y)$. It sends $V(xy)$ to itself and $(0,0)$, $(1,0)$, $(\alpha, 0)$, $(0,1)$ and $(0,\alpha)$ to $(0,0)$, $(\beta, 0)$, $(1, 0)$, $(0,\beta)$ and $(0,1)$, respectively. As $\alpha$ and $\beta$ are real, this automorphism lifts to an isomorphism $(W_{\alpha}, \rho_{\alpha}) \overset{\sim}{\rightarrow} (W_{\beta}, \rho_{\beta})$, as claimed. This proves the proposition.
\end{proof}

Proposition~\ref{prop: third assertion of main theorem} implies Statement~\ref{item: main theorem, 3} of Theorem~\ref{thm: new main theorem}.

\small{

\bibliographystyle{alpha}
\bibliography{refs}

}
\noindent

\bigskip\noindent
Anna Bot ({\tt annakatharina.bot@unibas.ch}),\\
Department of Mathematics and Computer Science, University of Basel, 4051 Basel, Switzerland

\end{document}